\documentclass[12pt]{amsart}
\usepackage{amsmath}
\usepackage{amssymb}
\usepackage{amsthm}
\usepackage{array}
\usepackage{xy}
\usepackage[pdftex]{graphicx}
\usepackage{hyperref}
\usepackage{color}
\usepackage{transparent}
\usepackage{latexsym}
\usepackage{ulem}
\usepackage[margin=1.5in]{geometry}
\usepackage{amsmath,scalefnt}
\usepackage{amsthm}
\usepackage{amssymb}
\usepackage{graphicx}
\usepackage{setspace}
\usepackage{hyperref}

\numberwithin{equation}{section}

\newtheorem{thm}{Theorem}[section]

\newtheorem{lem}[thm]{Lemma}

\newtheorem{prop}[thm]{Proposition}
\newtheorem{rem}[thm]{Remark}

\theoremstyle{plain}

\theoremstyle{plain}
\newtheorem{theorem}{Theorem} [section]

\theoremstyle{definition}

\theoremstyle{remark}
\newtheorem{remark}[theorem]{Remark}

\numberwithin{theorem}{section}
\numberwithin{equation}{section}
\numberwithin{figure}{section}

\begin{document}
\title[Regularity for the obstacle problem]{Free boundary regularity near the fixed boundary for the fully nonlinear obstacle problem}

\author[Emanuel Indrei]{Emanuel Indrei}


\def\signei{\bigskip\begin{center} {\sc Emanuel Indrei\par\vspace{3mm}Department of Mathematics\\  
Purdue University\\
West Lafayette, IN 47907, USA\\
email:} {\tt eindrei@purdue.edu}
\end{center}}

\makeatletter
\def\blfootnote{\xdef\@thefnmark{}\@footnotetext}
\makeatother

\date{}

\maketitle

\begin{abstract}
The interior free boundary theory for linear elliptic operators in higher dimensions was developed by Caffarelli \cite{MR0454350} in the low regularity context. In these notes, the up-to-the boundary free boundary regularity is discussed for nonlinear elliptic operators based on a different approach. 
\end{abstract}

\section{Introduction} 

Caffarelli proved that if $L$ is a linear uniformly elliptic operator and $u \ge 0$ solves 
\begin{equation}  \label{eq}
L(D^{2}u)=\chi_{\{u>0\}} \hskip .2in \text{in }B_{1}
\end{equation}

\noindent 
then for $x \in \Gamma=\partial \{u>0\} \cap B_{1}$ with positive Lebesgue density for $\{u=0\}$, i.e. satisfying 
$$
\displaystyle \liminf_{r\rightarrow 0^+}\frac{|B_r(x) \cap \{u=0\}|}{|B_r(x)|}>0 \hskip .2in \footnote{$|B_{r}(x) \cap \{u=0\}| \ge a r^n$ for $r>0$ small enough is usually assumed for convenience; $\displaystyle \limsup_{r\rightarrow 0^+}\frac{|B_r(x) \cap \{u=0\}|}{|B_r(x)|}>0$ is sufficient.}
$$
there is a Lipschitz function $g$ such that $\Gamma \cap B_s(x)$ admits a representation with respect to $g$ in a coordinate system for some $s>0$. The Lipschitz regularity can be improved to $C^1$ and higher regularity follows (up to analyticity) via a theorem of Kinderlehrer and Nirenberg \cite{MR0440187}. Caffarelli's theorem is optimal in the sense that there exists a solution when $L=\Delta$ for which there is a free boundary point with zero Lebesgue density for $\{u=0\}$ and in a neighborhood of the point the free boundary develops a cusp singularity and is not a graph in any system of coordinates \cite{MR0516201}.

In a recent work \cite{I}, the author proved that for solutions of \eqref{eq} with zero Dirichlet boundary data, with $L$ replaced by a convex fully nonlinear uniformly elliptic operator $F$, if $x \in \partial B_1 \cap \overline{\Gamma}$, then $\Gamma$ can be represented as the graph of a $C^1$ function in a neighborhood of $x$. 
There are two surprising differences between the interior and boundary result: first, there are no density assumptions in the boundary case (in particular, cusp-type singularities do not exist); second, there is an example which generates a free boundary which is $C^1$ with a specific Dini modulus of continuity for the free normal (see e.g. \cite[Remark 8.8]{MR2962060}). 

In his original approach, Caffarelli estimated pure second derivatives from below. The linear approach developed thereafter to handle regularity near the fixed boundary involves specific barrier constructions involving the operator and monotonicity formulas \cite{MR1392033, MR1950478, MR2281197}. The nonlinear method is based on understanding a maximal mixed partial derivative along a preferred direction.

In what follows, $F$ satisfies
\begin{itemize}
\item $F(0)=0$.
\item $F$ is uniformly elliptic with ellipticity constants $\lambda_{0}$, $\lambda_{1}>0$
such that
$$
\mathcal{P}^{-}(M-N)\le F(M)-F(N)\le\mathcal{P}^{+}(M-N),
$$
where $M$ and $N$ are symmetric matrices and $\mathcal{P}^{\pm}$
are the Pucci operators
$$
\mathcal{P}^{-}(M):=\inf_{\lambda_{0} \le N\le\lambda_{1}} \text{tr}(NM),\qquad\mathcal{P}^{+}(M):=\sup_{\lambda_{0}\le N\le\lambda_{1}}\text{tr} (NM).
$$
\item $F$ is convex and $C^1$.
\end{itemize}

Let $\Omega$ be an open set and $B_r^+=\{x: |x|<r, x_n>0\}$. A continuous function $u$ belongs to $P_r^+(0,M, \Omega)$ if $u$ satisfies in the viscosity sense

\begin{itemize}
\item $F(D^2 u)=\chi_\Omega$ in $B_r^+$;\\
\item $||u||_{L^\infty(B_r^+)} \le M$;\\
\item $u=0$ on $\{x_n=0\} \cap \overline{B_1^+}=:B'_{1}$.
\end{itemize}

In \cite{MR3513142} it was shown that $W^{2,p}$ solutions are $C^{1,1}$ (see also \cite{MR3198649, MR3542613} for the interior case). 
Furthermore, given $u \in P_r^+(0,M, \Omega)$, the free boundary is denoted by 
$
\Gamma=\partial \Omega \cap B_r^+.
$ 

A \textit{blow-up limit of $\{u_j\}$} $\subset P_1^+(0,M,\Omega)$ is a limit of the form $$\lim_{k \rightarrow \infty} \frac{u_{j_k}(s_kx)}{s_k^2},$$ where $\{j_k\}$ is a subsequence of $\{j\}$ and $s_k \rightarrow 0^+$. 

In \S \ref{S}, non-transversal intersection is shown for $\Omega=(\{u \neq 0\} \cup \{\nabla u \neq 0\}) \cap \{x_n>0\}$ and a problem in superconductivity is discussed in which $\Omega =\{\nabla u \neq 0\}) \cap \{x_n>0\}$; in \S \ref{S3}, $C^1$ regularity is proved when $u \ge 0$; last, some of the technical details are shown in the appendix \S \ref{S4}.

\section{Non-transversal intersection and classification of blow-up limits} \label{S}

One of the main results discussed in this section is the following. 
\begin{thm}\label{tt}
There exists $r_0>0$ and a modulus of continuity $\omega$ such that 
$$\Gamma(u) \cap B_{r_0}^+ \subset \{x: x_n \le \omega(|x'|)|x'|\}$$ for all $u \in P_1^+(0,M, \Omega)$ provided $0 \in \overline{\Gamma(u)}$ and $\Omega=(\{u \neq 0\} \cup \{\nabla u \neq 0\}) \cap \{x_n>0\}$.  
\end{thm}

If one varies the boundary data, then non-transversal intersection may not hold \cite[see Examples 3 \& 4]{MR2281197}. The difficulty in the fully nonlinear context is that monotonicity formulas are not available and a classification of blow-up limits requires a new approach: if blow-up limits are not half-space solutions, then a certain regularity property holds. More precisely:  

\begin{prop} \label{kee}
Suppose $\Omega=(\{u \neq 0\} \cup \{\nabla u \neq 0\}) \cap \{x_n>0\}$ and $\{u_j\} \subset P_1^+(0,M, \Omega)$. If $0 \in \overline{\{u_j \neq 0\}}$ and $\nabla u_j(0)=0$, then one of the following is true:\\
(i) all blow-up limits of $\{u_j\}$ at the origin are of the form $u_0(x)=bx_n^2$ for $b>0$;\\
(ii) there exists $\{u_{k_j}\} \subset \{u_j\}$ such that for all $R \ge 1$, there exists $j_R \in \mathbb{N}$ such that for all $j \ge j_R$, $$u_{k_j} \in C^{2,\alpha}(B_{\frac{Rr_{j}}{4}}^+),$$ where the sequence $\{r_{j}\}$ depends on $\{u_j\}$.  
\end{prop}

The proof relies on the fact that if not all blow-up solutions are half-space solutions, then one can construct a specific sequence producing a limit of the form $ax_1x_n+bx_n^2$.

\begin{prop} \label{th1}
Let $\{u_j\} \subset P_1^+(0,M, \Omega)$ and suppose $0 \in \overline{\{u_j \neq 0\}}$, $\{\nabla u_j \neq 0\} \cap \{x_n>0\} \subset \Omega$, $\nabla u_j(0)=0$. Then one of the following is true:\\
(i) all blow-up limits of $\{u_j\}$ at the origin are of the form $u_0(x)=b x_n^2$ for some $b >0$;\\ 
(ii) there exists a blow-up limit of $\{u_j\}$ of the form $ax_1x_n+bx_n^2$ for $a \neq 0$, $b \in \mathbb{R}$.
\end{prop}

\begin{proof}
Let $$N:= \limsup_{|x|\rightarrow 0, x_n>0} \frac{1}{x_n} \sup_{u \in \{u_j\}} \sup_{e \in \mathbb{S}^{n-2} \cap e_n^{\perp}} \partial_e u(x)$$ and consider a sequence $\{x^k\}_{k \in \mathbb{N}}$ with $x_n^k>0$, $u_{j_k} \in \{u_j\}$, and $e^k \in \mathbb{S}^{n-2} \cap e_n^{\perp}$ 
such that the previous limit is given by $$\lim_{k \rightarrow \infty} \frac{1}{x_n^k}  \partial_{e^k} u_{j_k}(x^k).$$ Note that $N<\infty$ by $C^{1,1}$ regularity for the class $P_1^+(0,M, \Omega)$ and the boundary condition (see \cite{MR3513142}). By compactness, $e^k \rightarrow e_1 \in \mathbb{S}^{n-2}$ (along a subsequence) so that up to a rotation, 
$$N= \lim_{k \rightarrow \infty} \frac{1}{x_n^k}  \partial_{x_1} u_{j_k}(x^k).$$   
\noindent Next, if $$\tilde u_j(x):= \frac{u_{k_j}(s_jx)}{s_j^2} \rightarrow u_0(x)$$ for some sequence $s_j \rightarrow 0^+$, where the convergence is in $C_{loc}^{1,\alpha}(\mathbb{R}_+^n)$ for any $\alpha \in [0,1)$, $u_0 \in C^{1,1}(\mathbb{R}_+^n)$ satisfies the following PDE in the viscosity sense
\begin{equation} \label{m2}
\begin{cases}
F(D^{2}u_0)=1 & \text{a.e. in }\mathbb{R}_+^n\cap\Omega_0\\
|\nabla u_0|=0 & \text{in }\mathbb{R}_+^n\backslash\Omega_0\\
u=0 & \text{on }\mathbb{R}_+^{n-1},
\end{cases}
\end{equation}
where $\Omega_0 =\{\nabla u_0 \neq 0 \} \cap \{x_n>0\}$. Note that 
\begin{equation} \label{part}
N \ge \lim_j \bigg| \frac{\partial_{x_i}  u_{k_j}(s_j x)}{s_j x_n}\bigg| = \lim_j \bigg| \frac{\partial_{x_i} \tilde u_j(x)}{x_n} \bigg| =\bigg|\frac{\partial_{x_i} u_0(x)}{x_n}\bigg|
\end{equation}
for all $i \in \{1,\ldots, n-1\}$. 
If $N=0$, then $\partial_{x_i} u_0=0$ for all $i \in \{1,\ldots,n-1\}$ so that $u_0(x)=u_0(x_n)$ and the conditions readily imply $u_0(x_n)=bx_n^2$. Since $N$ does not depend on the sequence $\{s_j\}$ it follows that in this case all blow-up limits have the previously stated form. Suppose that $N>0$, let $r_k=|x^k|$, and consider the re-scaling of $u_{j_k}$ with respect to $r_k$. Note that along a subsequence, $y^k:=\frac{x^k}{r_k} \rightarrow y \in \mathbb{S}^{n-1}$. By the choice of $r_k$, $$\lim_{k \rightarrow \infty} \frac{v(y^k)}{y_n^k}=\lim_{k \rightarrow \infty} \frac{\partial_{x_1} \tilde u_{k}(y^k)}{y_n^k}=\lim_{k \rightarrow \infty} \frac{\partial_{x_1} u_{j_k}(r_ky^k)}{r_ky_n^k} =N,$$ where $v= \partial_{x_1} u_0$. In particular, $$v(y)=Ny_n$$ and by an argument in \cite{MR3513142} (involving the boundary Harnack inequality), $u_0(x)=ax_1x_n+bx_n^2$ with $a \neq 0$. 
\end{proof}

\begin{proof}[Proof of Proposition \ref{kee}]
Either all blow-up limits are of the form $u_0(x)=bx_n^2$ or there exists a subsequence $$\tilde u_j(x)=\frac{u_{k_j}(r_jx)}{r_j^2}$$ producing a limit of the form $u_0(x)=ax_1x_n+bx_n^2$ for $a>0$ (up to a rotation). Let $c=c(a,b)$ be the constant from Lemma \ref{c} and note that since $\tilde u_j \rightarrow u_0$ in $C_{loc}^{1,\alpha}$, there exists $j_0=j_0(a, R) \in \mathbb{N}$ such that for every cylinder $S_{(\alpha, \beta)}(e_1)$ there exists $x \in S_{(\alpha, \beta)}(e_1) \cap B_R^+$ such that $|\nabla \tilde u_j(x)| \ge \frac{c}{2}$ for all $j \ge j_0$, where $R \ge 1$. Choose a constant  $C_0=C_0(a,b, R)>0$ such that $$C_0 \partial_{x_1} u_0 -u_0 \ge 0$$ in $B_R^+$ and $j_0' \ge j_0$ for which 
\begin{equation} \label{mon}
C_0 \partial_{x_1} \tilde u_j -  \tilde u_j \ge 0 \hskip .2 in \text{in $B_{\frac{R}{2}}^+$}
\end{equation} 
whenever $j \ge j_0'$ by Lemma \ref{m}. Now fix $j \ge j_0'$ and suppose $z \in \Gamma_i(\tilde u_j) \cap B_{\frac{R}{2}}^+$. Then there exists a ball $B \subset \text{int} \{\tilde u_j=0\} \cap B_{\frac{R}{2}}^+$ and a cylinder $S$ in the $e_1$- direction generated by $B$. Now select $x \in S \cap B_R^+$ for which $|\nabla \tilde u_j(x)|>0$ and $-R<x_1<-R/2$. In particular, there exists a small ball around $x$, say $\tilde B$ such that $F(D^2 \tilde u_j)=1$ in $\tilde B$ and one may assume $\tilde B \subset \{\tilde u_j \neq 0\}$. Note that $\tilde B$ is contained in the cylinder $S$ and let $E_t=\tilde B+te_1$ for $t \in \mathbb{R}$. If $t>0$ is such that $\overline{E_t} \cap \{\tilde u_j=0\} \neq \emptyset$, and for all $0\le s<t$, $E_s \cap \{\tilde u_j=0\} = \emptyset$, choose $y \in \overline{E_t} \cap \{\tilde u_j=0\}.$ If $\tilde u_j > 0$ in $\tilde B$, then by \eqref{mon} it follows that $\tilde u_j$ is strictly positive at a point in $\{\tilde u_j=0\}$, a contradiction. Thus $\tilde u_j < 0$ in $\tilde B$. By convexity of $F$
$$a_{kl} \partial_{kl} \tilde u_j \ge 0 \hskip .1in \text{in $E_t$}.$$ Since $0=\tilde u_j(y)>\tilde u_j(x)$ for $x \in E_t$ and $y$ satisfies an interior ball condition, then Hopf's lemma implies that $\frac{\partial}{\partial n} \tilde u_j(y)>0$, where $n$ is the outer normal to the ball at $y$. If there exists $z \in B_\delta(y)$ such that $\tilde u_j(z)>0$, then this contradicts the monotonicity, if $\delta>0$ is sufficiently small: $\overline{E_{\eta}} \subset B \subset int \{\tilde u_j=0\}$ for $\eta>0$ large enough and since $\tilde u_j(z)>0$, the monotonicity \eqref{mon} implies that $\tilde u_j(z+e_1s)>0$, for some $s>0$ such that $z+e_1s \in \{\tilde u_j=0\}$. Hence, $\tilde u_j\le 0$ on $B_\delta(y)$ and thus $\nabla \tilde u_j(y)=0$, a contradiction.      
The conclusion is that for $j \ge j_0'$, $$\Gamma_i(\tilde u_j) \cap B_{\frac{R}{2}}^+=\emptyset.$$ In particular, $(B_{\frac{R}{2}}^+ \setminus \Omega_j)^o=\emptyset$ and non-degeneracy implies that $|B_{\frac{R}{2}}^+\setminus \Omega_j|=0$. Thus the $C^{1,1}$ function $\tilde u_j$ satisfies $F(D^2 \tilde u_j)=1$ in $B_{\frac{R}{2}}^+$ in the viscosity sense and the up to the boundary Evans-Krylov theorem (see e.g. \cite{Saf}) implies that $\tilde u_j \in C^{2,\alpha}(B_{\frac{R}{4}}^+).$ In particular, $u_{k_j} \in C^{2,\alpha}(B_{\frac{Rr_j}{4}}^+).$ 
\end{proof}

\begin{thm}  \label{cul}
Suppose $u \in P_1^+(0,M, \Omega)$ and $\Omega=(\{u \neq 0\} \cup \{\nabla u \neq 0\}) \cap \{x_n>0\}$. If $0 \in \overline{\{u \neq 0\}}$ and $\nabla u(0)=0$, then the blow-up limit of $u$ at the origin has the form $$u_0(x)=ax_1x_n+bx_n^2$$ for $a, b \in \mathbb{R}$.
\end{thm}

\begin{proof}
By Proposition \ref{kee}, either $u_0(x)=bx_n^2$ or $D^2u(0)$ exists and the rescaling of $u$ is given by $$u_j(x)=\frac{u(r_j x)}{r_j^2}=\langle x, D^2u(0)x\rangle + o(1).$$ Since $u_0(x',0)=0$ for $x' \in \mathbb{R}^{n-1}$, it follows that $u_0$ has the claimed form (up to a rotation). 
\end{proof}

\begin{thm}\label{cula}
Suppose $\Omega=(\{u \neq 0\} \cup \{\nabla u \neq 0\}) \cap \{x_n>0\}$, $0 \in \overline{\Gamma}$, and $\{u_j\} \subset P_1^+(0,M, \Omega)$. Then the blow-up limit of $\{u_j\}$ at the origin has the form $$u_0(x)=bx_n^2$$ for $b>0$.
\end{thm}

\begin{proof}
By Proposition \ref{kee}, either $u_0(x)=bx_n^2$ or there exists a subsequence $$u_{k_j}(x)\in C^{2,\alpha}(B_{\frac{Rr_{j}}{4}}^+)$$ which contradicts that $F$ is continuous (consider a sequence of points approaching the free boundary from the set where the equation is satisfied with the right-hand-side being equal to one and from the complement). 
\end{proof}

\begin{remark}
There exist global solutions which are not blow-up solutions (at contact points).
\end{remark}

\begin{proof}[proof of Theorem \ref{tt}]
It suffices to show that for any $\epsilon>0$ there exists $\rho_\epsilon>0$ such that $\Gamma(u) \cap B_{\rho_\epsilon}^+ \subset B_{\rho_\epsilon}^+ \setminus \mathcal{C}_\epsilon$, where $\mathcal{C}_\epsilon=\{x_n > \epsilon |x'| \}$. If not, then there exists $\epsilon>0$ such that for all $k \in \mathbb{N}$ there exists $u_k \in P_1^+(0,M, \Omega)$ with 
\begin{equation} \label{cont}
\Gamma(u_k) \cap B_{1/k}^+ \cap \mathcal{C}_\epsilon \neq \emptyset, 
\end{equation}
where $0 \in \overline{\Gamma(u_k)}.$ If all blow-ups of $\{u_k\}$ are half-space solutions. Let $x_k \in \Gamma(u_k) \cap B_{1/k}^+ \cap \mathcal{C}_\epsilon$ and set $y_k=\frac{x_k}{r_k}$ with $r_k=|x_k|$. Consider $\tilde u_k(x)=\frac{u_k(r_kx)}{r_k^2}$ so that $y_k \in \Gamma(\tilde u_k)$, $\tilde u_k \rightarrow bx_n^2$, $y_{k} \rightarrow y \in \partial B_1 \cap C_\epsilon$ (up to a subsequence), and $y \in \Gamma(u_0)$, a contradiction. Second, select a subsequence $\{u_{k_j}\}$ of $\{u_k\}$ such that for all $j \ge j_2$, $u_{k_j} \in C^{2,\alpha}(B_{\frac{r_{j}}{2}}^+)$, where $j_2 \in \mathbb{N}$ and the sequence $\{r_{j}\}$ depends on $\{u_k\}$. Since $0 \in \overline{\Gamma(u_{k_j})},$ there exists $$x_j \in \Gamma(u_{k_j}) \cap B_{\frac{r_{j}}{2}}^+$$ which contradicts the continuity of $F$ (consider a sequence of points approaching the free boundary from the set where the equation is satisfied with the right-hand-side being equal to one and from the complement).    
\end{proof}

\subsection{An obstacle problem in superconductivity} 
Equations of the type
$$
F(D^2u,x)=g(x,u)\chi_{\{\nabla u \neq 0\}}
$$
have been investigated in \cite{MR1897393} and are based on physical models, e.g. the stationary equation for the mean-field theory of superconducting vortices when the scalar stream is a function of the scalar magnetic potential \cite{MR1349309, MR1388106, MR1664550}. It is shown that in certain configurations in two dimensions, the set $\{\nabla u=0\}$ is convex. 
In a recent paper \cite{I2}, the author proved non-transversal intersection for $\Omega = \{\nabla u \neq 0\} \cap \{x_2>0\}$.
If $\{u<0\}$ has sufficiently small density, non-transversal intersection follows from the techniques discussed above without a dimension restriction: suppose
$$	 
\frac{|\{u<0\}\cap B_r^+|}{|B_r^+|} \rightarrow 0
$$
as $r\rightarrow 0^+$.
A limit of the form $$u_0(x)=\lim_{k \rightarrow \infty} \frac{u_{j_k}(s_kx)}{s_k^2}$$ satisfies $u_0 \ge 0$ and therefore cannot be $ax_1x_n+bx_n^2$ for $a \neq 0$. In particular, it must be a half-space solution by Proposition \ref{th1} and the non-transversal intersection follows as before. The assumption on the negativity set appeared in \cite{MR2065018} where the authors considered the non-transversal intersection subject to additional assumptions on the operator and solution.

\section{$C^1$ regularity} \label{S3}

In the physical case when $u \ge 0$, the free boundary is $C^1$ without density assumptions. 

\begin{thm} \label{c1r}
Let $u \in P_1^+(0, M, \Omega)$ be non-negative, $\Omega=(\{u \neq 0\} \cup \{\nabla u\neq 0\}) \cap \{x_n>0\}$, and $0 \in \overline{\Gamma(u)}$. There exists $r_0>0$ such that $\Gamma$ is the graph of a $C^1$ function in $B_{r_0}^+$.
\end{thm}

\begin{proof} 
First, for any $\epsilon>0$ there exists $r(\epsilon, M)>0$ such that if $x^0 \in \Gamma(u) \cap B_{1/2}^+$ and $d=x_n^0<r,$ then 
$$ 
\sup_{B_{2d}^+(x^0)} |u-h| \le \epsilon d^2, \hskip .2in \sup_{B_{2d}^+(x^0)} | \nabla u - \nabla h| \le \epsilon d, 
$$
where $$h(x)= b[(x_n-d)^+]^2,$$ and $b>0$ depends on the ellipticity constants of $F$. If not, then there exists $\epsilon>0$, non-negative $u_j \in P_1^+(0,M, \Omega)$, and $x^j \in \Gamma(u_j) \cap B_{1/2}^+$ with $d_j=x_n^j \rightarrow 0$, for which 
$$\sup_{B_{2d_j}(x^j)^+} |u_j- b[(x_n-d_j)^+]^2|>\epsilon d_j^2,$$ or 
$$\sup_{B_{2d_j}(x^j)^+} |\nabla u_j- 2b(x_n-d_j)^+|>\epsilon d_j.$$   
Let $\tilde u_j(x)=\frac{u_j((x^j)'+d_jx)}{d_j^2}$ so that in particular 
$$||\tilde u_j - h||_{C^1(B_2^+(e_n))} \ge \epsilon,$$ where $h(x)=b[(x_n-1)^+]^2$. Since $\tilde u_j(e_n)=|\nabla \tilde u_j(e_n)|=0,$ the $C^{1,1}$ regularity of $\tilde u_j$ implies that $|\tilde u_j(x)| \le C|x-e_n|^2$. By passing to a subsequence, if necessary, $$\tilde u_j \rightarrow u_0$$ where $u_0 \in C^{1,1}(\mathbb{R}_+^n)$ satisfies the following PDE in the viscosity sense
\begin{equation} \label{m3}
\begin{cases}
F(D^{2}u_0)=1 & \text{a.e. in }\mathbb{R}_+^n\cap\Omega_0,\\
|\nabla u_0|=0=u_0 & \text{in }\mathbb{R}_+^n\backslash\Omega_0,\\
u_0=0 & \text{on }\mathbb{R}_+^{n-1}.
\end{cases}
\end{equation}   
Now let $$N=\limsup_{|x|\rightarrow 0, x_n>0} \frac{1}{x_n} \sup_{u\in P_1^+\cap \{u \ge 0\}} \sup_{e \in \mathbb{S}^{n-2} \cap e_n^{\perp}} \sup_{y \in \overline{B_{1/2}^+} \cap \{x_n=0\}} \partial_e u(x+y)$$ and note that $N<\infty$ by $C^{1,1}$ regularity and the boundary condition: for any 
$e \in \mathbb{S}^{n-2} \cap e_n^{\perp}$ and $y \in \overline{B_{1/2}^+} \cap \{x_n=0\}$, it follows that $\partial_{e} u(x'+y)=0$. Furthermore, 

\begin{equation} \label{part}
N \ge \lim_j \bigg| \frac{\partial_{x_i}  u_j(d_j x+(x^j)')}{d_j x_n}\bigg| = \lim_j \bigg| \frac{\partial_{x_i} \tilde u_j(x)}{x_n} \bigg| =\bigg|\frac{\partial_{x_i} u_0(x)}{x_n}\bigg|
\end{equation}
for all $i \in \{1,\ldots, n-1\}$. In particular, let $v= \partial_{x_1} u_0$ so that in $\mathbb{R}_+^n$, 
\begin{equation} \label{ine}
|v(x)| \le Nx_n. 
\end{equation}
If $N=0$, then $\partial_{x_i} u_0=0$ for all $i \in \{1,\ldots,n-1\}$ and therefore $u_0(x)=u_0(x_n)$. Since $e_n$ is a free boundary point, it follows that $u_0=h$, a contradiction. Thus $N>0$  
and there is a sequence 
$\{x^k\}_{k \in \mathbb{N}}$ with $x_n^k>0$, $u_k \in P_1^+(0,M, \Omega)$, $u_k \ge 0$, $y^k \in \overline{B_{1/2}^+} \cap \{x_n=0\}$, and $e^k \in \mathbb{S}^{n-2} \cap e_n^{\perp}$ such that $$N=\lim_{k \rightarrow \infty} \frac{1}{x_n^k}  \partial_{e^{k}} u_k(x^k+y^k).$$ 
By compactness, $e^k \rightarrow e_1 \in \mathbb{S}^{n-2}$ (along a subsequence) so that up to a rotation, 
$$N= \lim_{k \rightarrow \infty} \frac{1}{x_n^k}  \partial_{x_1} u_k(x^k+y^k).$$ Let 
$$\tilde u_k(x)=\frac{u_k(y^k+r_kx)}{r_k^2},$$ where $r_k=|x^k|$, $z^k=\frac{x^k}{r_k}$, and note that along a subsequence $z^k \rightarrow z \in \mathbb{S}^{n-1}$ and $\tilde u_k \rightarrow u_0$. It follows that $\partial_{x_1}u_0(z)=Nz_n$ and proceeding as in \cite{MR3513142} one deduces that $u_0(x)=ax_1x_n+cx_n+\tilde b x_n^2$ for $a \neq 0$ and $c, \tilde b \in \mathbb{R}$, contradicting that $u \ge 0$. This implies that in a neighborhood of the origin, there is a cone of fixed opening that can be placed below and above each free boundary point; therefore, the free boundary is Lipschitz continuous and thus $C^1$ by interior results \cite[Theorem 1.3]{MR3198649}. Since the intersection of $\Gamma$ and the origin occurs non-transversally, and 
$$ 
\sup_{B_{2d}^+(x^0)} |u-h| \le \epsilon d^2, \hskip .2in \sup_{B_{2d}^+(x^0)} | \nabla u - \nabla h| \le \epsilon d, 
$$
the aperture of the cones can be taken arbitrarily close to $\pi$.     
\end{proof}

{\bf Acknowledgement} The author wishes to thank Donatella Danielli and Irina Mitrea for organizing the AMS Special Session ``Harmonic Analysis and Partial Differential Equations" at Northeastern University.

\section{Appendix} \label{S4}

\begin{lem} \label{m}
Let $u \in P_r^+(0,M, \Omega)$ where $\{u \neq 0\} \subset \Omega$, $e \in \mathbb{S}^{n-2} \cap e_n^{\perp}$, and suppose there exist non-negative constants $\epsilon_0, C_0$ such that $C_0 \partial_e u -u \ge -\epsilon_0$ in $B_r^+$. Then there exists $c=c(n, \Lambda, r)>0$ such that if $\epsilon_0 \le c$, then $C_0 \partial_e u - u \ge 0$ in $B_{\frac{r}{2}}^+$. 
\end{lem}

\begin{proof}
By convexity of $F$, there exist measurable uniformly elliptic coefficients $a_{ij}$ such that 
$$ 
F(D^2 u(x+he))-F(D^2 u(x)) \ge a_{ij}(\partial_{ij} u(x+he)-\partial_{ij} u(x)) 
$$
if $x \in \Omega$ provided $h$ is small enough. Therefore, 

$$0 \ge a_{ij} \partial_{ij} \partial_e u \hskip .1in \text{in $\Omega$}.$$ Convexity also yields 

$$a_{ij} \partial_{ij} u \ge F(D^2 u(x))-F(0) = 1 \hskip .1in \text{in $\Omega$}.$$ Suppose now that there exists $y \in B_{\frac{r}{2}}^+$ for which  $C_0 \partial_e u(y) - u(y) < 0.$ Let $w(x)=C_0 \partial_e u(x)-u(x)+\frac{|x-y|^2}{2n\Lambda}$. Since $\lambda Id \le (a_{ij}) \le \Lambda Id$, it follows by the above that $L w \le 0$ in $\Omega$ where $L=a_{ij}\partial_{ij}$. The maximum principle implies $\min_{\partial(\Omega \cap B_r^+)} w=\min_{\Omega \cap B_r^+} w < 0$. Note that $w\ge 0$ on $\partial \Omega$ and likewise on $\{x_n=0\}$. Therefore, the minimum occurs on $\partial B_r$ and thus $0 > -\epsilon_0+\frac{1}{8n \Lambda}r^2$, a contradiction if $\epsilon_0$ is small enough.  
\end{proof}

\begin{rem} \label{rem1}
One may take $\epsilon_0=c r^2$, where $c>0$ depends only on the dimension and ellipticity constants of $F$.  
\end{rem}

\begin{rem}
If $u \ge 0$, then $\partial_{e_n} u \ge 0$ on $\{x_n=0\} \cap B_r$ and Lemma \ref{m} holds therefore in this case for all $e \in \mathbb{S}^{n-1}$ such that $e \cdot e_n \ge 0$. 
\end{rem}

\begin{lem} \label{c}
Let $u_0(x)=ax_1x_n+bx_n^2$ with $a \neq 0$ and $R\ge1$. Then there exists $c=c(a,b)>0$ such that $$\inf_D |\nabla u_0(x)| \ge c,$$ where $D=\{x=(x_1,x'',x_n): R > |x| > R/2, |x''| \le \delta(R)\}$ for some $\delta(R)>0$.   
\end{lem}

\begin{proof}
Note $|\nabla u_0(x)|^2=a^2x_n^2+a^2x_1^2+2abx_1x_n+4b^2x_n^2$ so that if $|x_n|>\frac{1}{3}$, then $|\nabla u_0(x)|^2 \ge \frac{a^2}{9}$. If $|x_n| \le \frac{1}{3}$, then for points that satisfy $|x''| \le \sqrt{\frac{5}{72}R}$, where $x''=(x_2,x_3,\ldots,x_{n-1})$, it follows that 
$$
x_1^2 > \frac{5}{72} R^2.
$$
If $b \neq 0$, let $\epsilon^2 \in (\frac{1}{a^2+4b^2}, \frac{1}{b^2}).$ Then 
\begin{align*}
|\nabla u_0(x)|^2&\ge(a^2+4b^2-\frac{1}{\epsilon^2})x_n^2+(a^2-\epsilon^2a^2b^2)x_1^2\\
&> (a^2-\epsilon^2a^2b^2)(\frac{5}{72} R^2).
\end{align*} 
\end{proof}

\begin{lem} \label{d}
Let $u_0(x)=ax_1x_n+bx_n^2$ with $a > 0$ and $R \ge 1$. Then there exists $C_0=C_0(a,b,R)>0$ such that 
$$
C_0\partial_{x_1} u_0(x)-u_0(x) \ge 0
$$
in $B_R^+$.
\end{lem}

\begin{proof}
The condition is equivalent to $ax_n(C_0-x_1) \ge b x_n^2$. Since $x_1 \le R$ and $0\le x_n \le R$, it follows that any $C_0 \ge \frac{b}{a}R+R$ satisfies the condition.
\end{proof}

\bibliographystyle{alpha}

\bibliography{ngonref2}

\signei

\end{document}